\documentclass[oneside,leqno,12pt]{amsart}

\usepackage{fullpage}
\usepackage{amsfonts}
\usepackage{amsmath}
\usepackage{amsthm}
\usepackage{dsfont}
\usepackage{amsthm}
\usepackage{amssymb} 
\usepackage{amsxtra}     
\usepackage{epsfig}
\usepackage{verbatim}
\usepackage{color}
\usepackage{enumitem}

\newtheorem{theorem}{Theorem}

\newtheorem{corollary}[theorem]{Corollary}
\newtheorem{lemma}[theorem]{Lemma}

\newtheorem{remark}[theorem]{Remark}
\numberwithin  {equation}{section}

\newcommand{\R }{\mathbb{R}}
\renewcommand{\L }{\mathcal{L}}

\renewcommand{\phi}{\varphi }
\begin{document}
\title[] { Sublinear elliptic problems under radiality. Harmonic $NA$ groups and Euclidean spaces. }
\author{ Ewa Damek, Ghardallou Zeineb\\
{\tiny{\emph{Departement of Mathematics, Wroclaw University, 50-384 Wroclaw, Plac Grunwaldzki 2/4, Poland
\\Faculty of sciences of Tunis, Department of Mathematical Analysis and Applications, University Tunis El Manar, LR11ES11, 2092 El Manar, Tunis, Tunisia}}}} 
 
\maketitle\footnote{ \noindent The first author was supported by the NCN Grant UMO-2014/15/B/ST1/00060.}





\textbf{Abstract}
Let $\L $ be the Laplace operator on $\R ^d$, $d\geq 3$ or the Laplace Beltrami operator on the harmonic $NA$ group (in particular on a rank one noncompact symmetric space). For the equation $ \L u  - \varphi(\cdot,u)=0$ we
give necessary and sufficient conditions for the existence of entire
 bounded or large solutions under the hypothesis of
radiality of $\varphi$ with respect to the first variable. 
A Harnack-type inequality for positive continuous solutions is also proved.\\


\textbf{Mathematic Subject Classification (2010)}: Primary: 35J08, 35J61, 53C35, 58J05; Secondary: 22E30, 43A80, 53C35.

\section{Introduction}

Let $\L $ be the Laplace operator on $\R ^d$ or the Laplace-Beltrami operator  on a harmonic $NA$ group \footnote{Such groups are also called
Damek-Ricci spaces and the family includes the rank one noncompact symmetric spaces. For their basic properties see Section \ref{NAgroups}.}. We are interested in nonnegative entire solutions to 
\begin{equation}\label{subproblem}
\L u(x) - \varphi (x, u(x))=0, \ x\in S,
\end{equation}
where $S$ is either Euclidean $\R^d$ or harmonic $NA$ and 
 $\varphi : S\times [0,\infty )\to [0,\infty )$ is radial with respect to the first variable. As a Riemannian manifold harmonic $NA$ group is diffeomorphic to $\R ^d $ with an appropriate Riemannian metric and it has global geodesic coordinates. Therefore, 
radiality means that  $\phi (x,t)=\phi (d(0,x),t)$ where $d$ is the  Riemannian distance on harmonic $NA$ or Euclidean distance on $\R ^d$. The behavior of radial objects is very similar on both spaces and so we may apply a unified approach.

Our aim is to give necessary and sufficient conditions for existence of entire
solutions bounded or ``large'' under the hypothesis of radiality. 
A solution $u$ to
\eqref{subproblem} is called {\it large } if $u(x)\to \infty $ when $d(x,0)\to \infty $. The characterization we obtain is easily formulated in terms of $\phi $, the condition is straightforward to check and the whole problem is nicely linked to the geometry of the harmonic space $NA$. Since non compact rank one symmetric spaces are among harmonic $NA$ groups, the results
we obtain apply also to the Laplace Beltrami operator there. In particular,
the equation
$$
Lu-u^{\gamma },\quad \gamma \leq 1$$
on rank a one non compact symmetric space does not have a bounded solution, but it has a radial large solution.

{Large solutions i.e. the boundary blow-up problems are of considerable interest due to its several scientific applications in different fields. Such problems arise in the study of Riemannian geometry \cite{Cheng-Ni}, non-Newtonian fluids \cite{Astrita-Marrucci}, the subsonic motion of a gas \cite{Pohzaev} and the electric potential in some bodies \cite{Lazer-Mckenna}.

The vast majority of papers studying similar problems consider the equation $\Delta u + g(x,u)=0$, for the Lapace operator $\Delta $ on $\R ^d$. Variety of hypotheses on $g$ have been supposed and various questions have been asked (see e.g. \cite{D-G} for comprehensive references). In general the semilinear term is not nonnegative or radial with respect to the first variable.
Let us abandon for a moment radiality of $\phi $ but stick to $\phi \geq 0$.
Then under mild conditions bounded entire solutions of \eqref{subproblem} can be fully characterized and linked to the bounded $\L- $harmonic functions (see \cite{Ghardallou2} for more general elliptic operators or Section \ref{secbounded}). Secondly, entire bounded and large solutions cannot exist at the same time \cite{D-G}. Therefore, a neat characterization of existence of this two types of solutions is essential.

Throughout this paper, }we assume as little regularity of $\phi $ as possible but we want $\phi $ to be  nondecreasing with respect to the second variable. This allows us to apply a quite general potential theoretical approach \cite{ BH,ELMabrouk2004,Ghardallou2}. To summarize, our assumptions are:

\medskip
\begin{description}
	\item[($H_1$)] For every $t_0\in [0,+\infty )$, $x\mapsto \varphi(x,t_0)\in K_d^{loc}(\Omega)$
	i.e. it is locally in the Kato class. \footnote{We say that a Borel measurable function
		$p$ on $\Omega$ is locally in the Kato class in $\Omega$ if $$ \lim\limits_{\alpha\to
			0}\sup_{x\in D}\int_{D\cap(|x-y|\leq \alpha)}\frac{|p(y)|}{|x-y|^{d-2}}\,dy=0$$ for
		every open bounded set satisfying  ${D}\subset \Omega $.}
	\item[$(H_2)$] For every $x_0\in \Omega$,  $t\mapsto \varphi(x_0, t )$ is   continuous
	increasing on $[0,+\infty )$
	
	\item[$(H_3)$]$\varphi(x,t)=0$ for every $x\in \Omega$ and $t\leq 0$.  

\end{description}

\medskip
 $(H_1)$  makes $\phi $ locally integrable against the fundamental solution to $\L $ which plays a crucial role in our approach. $(H_3)$ is a technical extension of $\phi $ to $(-\infty , 0]$ needed as a tool. $(H_2)$ is motivated by the previous results for $\Delta $ \cite{ElmabroukHansen,Lair2007, LairWood2000}, where $\phi (x,t)= p(x)\psi (t)$ and $\psi (t)=t^{\gamma}$ or
$\psi $  satisfies so called Keller-Osserman condition. Moreover, under $(H_1)-(H_3)$ we have full description of bounded solutions provided the Green potential of $\phi $ is finite, see \eqref{finitepotential} in Section \ref{secbounded}.

Equation \eqref{subproblem} with $\L $ being the Laplace operator $\Delta $ on $\R ^d $ has been studied with the semilinear part of the form $\phi (x,t)=p(x)u^{\alpha }+q(x)u^{\beta }$, $p,q$ positive functions, $0<\alpha \leq \beta $ or, more generally, $\phi (x,t)=p(x)f(u)+q(x)g(u)$ \cite{AACH, Bandle,ELMabrouk2004,ElmabroukHansen,Lair2007,Lair, LM,LairWood2000,Ahmed}. Then conditions for existence or nonexistence of large solutions are formulated in terms of $\int _0^{\infty }rh(r)dr$ being finite or infinite, where $h(r)$ is one of the functions: $\max _{|x|=r}p(x)$, $\max _{|x|=r}q(x)$, $\min _{|x|=r}p(x)$, $\min _{|x|=r}q(x)$, \cite{ElmabroukHansen,Lair}. So it makes sense to consider first the radial case. If 
 $\varphi (x,t)=p(x)t^{\gamma }, \gamma <1$, $p$ is radial, locally bounded then 
$$\Delta u - p(x)u^{\gamma }=0$$
has an entire large solution if and only if $$ \int_{0}^{\infty} rp(r) dr= \infty ,$$ \cite{ElmabroukHansen}, see also \cite{LairWood2000}. Otherwise \eqref{ancientproblem} has an entire bounded solution.
We obtain here a similar characterization under weaker conditions on $\phi $ both on $NA$ and $\R ^d$ which generalizes also existing results for the Euclidean case.
 
For radial $\varphi $ we write $\varphi (r,t)=\varphi (x,t)$, where $r=d(x,0)$. Let 
\begin{equation}\label{IvarphiR^d}
I\varphi (\cdot,c)=\int _0^{\infty}r\varphi (r,c)\ dr,
\end{equation}
for $\phi $ on Euclidean $S$, while for $\varphi $ on harmonic $S $ we write
\begin{equation}\label{IvarphiNA}
I\varphi (\cdot,c)=\int _0^{\infty}\varphi (r,c)\ dr.
\end{equation}
We prove that \eqref{subproblem} has always an entire solution. Being bounded or large depend essentially on the growth of $\varphi(\cdot,c)$ measured by $I\varphi(\cdot,c)$ (see Section
\ref{bounded-large-radial-case}): 

\begin{theorem}\label{texistence2}
	Suppose that $\varphi (s,t)=p(s)\psi (t)$, satisfies  $(H_2)$-$(H_3)$, $p\in \mathcal{L}^{\infty}_{loc}(S)$ is radial and there is $c>0$ such that
	\begin{equation}\label{psi}
\psi (t)\leq c(t+1).
\end{equation}
		Then \eqref{subproblem} has always an entire radial solution. Moreover, there is a bounded
	solution if and only if $I(p)<\infty $ and there is a large solution if and only if $I(p)=\infty $.
	
\end{theorem}

\begin{theorem}\label{texistence}
	Suppose that ($H_1$)-($H_3$) are satisfied and $\varphi$ is radial with respect to the first variable. Suppose further that
	
	\medskip \noindent
	($H_4$): For every $x_0\in \Omega$,  $t\mapsto \varphi(x_0, t )$ is concave on $[0,+\infty[$.
	
	\medskip
	Then \eqref{subproblem} has always an entire radial solution. Moreover, there is a
	bounded solution if and only if there exists $c>0$ such that, $I\varphi (\cdot ,c)<\infty $ and there is a large
	solution if and only if for every $c>0$, $I\varphi (\cdot ,c)=\infty $.
\end{theorem}

Theorems \ref{texistence2} and \ref{texistence} show that large solution and bounded non zero solutions cannot exist at the same time as it is already proved in \cite{D-G} for general elliptic operator.
Also, Theorems \ref{texistence2} and \ref{texistence} give a nice characterization of existence of bounded and large
solutions. Theorem \ref{texistence2} and \ref{texistence} follow directly from Theorems \ref{bounded-solution}, \ref{large-solution-in-radial-case}, \ref{largenoconcavity}. The techniques developed before in \cite{D-G}, \cite{Ghardallou2} allow us to go beyond nonlinearity being $t^{\gamma }$ or even concavity. 

Notice that sublinear $\psi $, \eqref{psi} satisfies the ``so called'' Keller-Osserman condition \cite{keller,Osserman}
\begin{equation}\label{KO}
 \int_{1}^{\infty} \Big[\int_{0}^{s} \psi(t) dt  \Big]^{-\frac{1}{2}} ds=\infty
\end{equation}
that somehow fits well into the large solution problem. {Namely for $p=1$ and $\psi$ positive continuous Keller \cite{keller} and Osserman \cite{Osserman} proved \eqref{KO} is a necessary and sufficient condition to have an entire spherically symmetric solution to \begin{equation}\label{ancientproblem}
	\Delta u - p \psi (u) =0.
	\end{equation} 
Lair \cite{Lair2007} extended this to the case where function $p$ is continuous positive spherically symmetric and $\psi $ is nonnegative continuous nondecreasing vanishing at origin, satisfying \eqref{KO}. They proved that \eqref{ancientproblem} has always an entire solution, but there is no such complete
 characterization of both cases (bounded or large) as we have here under the stronger assumption \eqref{psi}.}

The difference in 
$I\varphi(\cdot,c)$ for $\R^d$ \eqref{IvarphiR^d} and 
for $NA$ \eqref{IvarphiNA} is due to the properties 
of the fundamental solution $G$ to $\L $ on each space. 
More precisely, in both spaces 
$G$ is radial. On $\R^d$, 
$G(x)= a_d |x|^{-d+2}$ where $a_d$ is a constant depending only on the dimension. For $NA$, the precise estimate for $G$ is obtained in Theorems \ref{estimation-Green-r-big} and \ref{estimation-Green-small-r}. 
Using radiality on $NA$ group as it is usually done on $\R^d$, we obtain that 
\eqref{finitepotential} is equivalent to $I\varphi(\cdot,c)$ being finite for some $c$. 
Properties of harmonic $NA$ groups as well as of $G$ are discussed in Section 3.

To prove our results we need a Harnack-type inequality. It holds for general elliptic operators and it is proved in the last section.
More precisely, we have the following theorem
\begin{theorem}\label{Mainthe-Harnack-ineq}
	Let $L$ be a second order elliptic operator with smooth coefficients defined on a domain $\Omega \subset \R ^d$, $d\geq 3$, $L1\leq 0$. 
	Assume that $\varphi:\Omega\times[0,+\infty )\to[0,+\infty )$ is a positive function increasing with respect to the second variable satisfying 
	$$(H_1')~~\varphi(x,u)\leq p(x)(u+1), ~~where~~ p\in \mathcal{K}_d^{loc}(\Omega)~.
	\footnote{If ($H_3$) and ($H_4$) are satisfied then $\varphi (x,t)\leq t
		\varphi (x,1)$ for $t\geq 1$, which means that ($H_1$)-($H_4$) imply ($H'_1$).}
	$$
	
	Then for every compact set $R$ in $\Omega$ there exists a
	constant $C_R>0$ such that for every positive solution of the equation \begin{equation}\label{Main-problem}
	Lu -\varphi(\cdot,u)=0, \mbox{ in $\Omega$}
	\end{equation} we have
	$$\sup_{x\in R}u(x)\leq C_R (1+\inf_{x\in R}u(x)).$$
	
\end{theorem}
Theorem \ref{Mainthe-Harnack-ineq} was proved in \cite{B.H.H} for  
$\Delta $ and $\varphi
(x,u)=p(x)u^{\gamma }$, $0<\gamma <1$. It turns out that potential theory for $L$ and a
sublinearity of $\varphi $ are completely sufficient to follow their approach.

At last, the authors want to express
their gratitude to {Krzysztof Bogdan, Konrad Kolesko, Mohamed Selmi and Mohamed Sifi} for their helpful and
kindly suggestions.

\section{Bounded and large solution under radiality}\label{bounded-large-radial-case}\label{secbounded}


This section contains our main results about bounded and large solutions. Some  theorems and lemmas holds for general elliptic operators with smooth coefficients. So, as before, we will use for them notation $L$, while $\L $ is reserved for the Laplace Beltrami operator on $S$.

\subsection{Elliptic operators with smooth coefficients}

Let $\Omega $ be an open domain. We say that $\Omega $ is {\it Greenian}
if there is a function $G_{\Omega }(x,y)$ smooth on $ \Omega \times \Omega \setminus \{ (x,x): x\in
\Omega \}$ such that for every $y\in \Omega$
\begin{equation}\label{delta}
L G_{\Omega }(\cdot , y)=-\delta _y,\quad \mbox{in the sense of distributions,}
\end{equation}
and
\begin{equation}\label{potential}
G_{\Omega }(\cdot , y),\quad \mbox{is a potential},
\end{equation}
i.e. every positive $L $-harmonic function $h$ such that $h(x)\leq G_{\Omega }(x,y)$ is equal 0. We
write
$$
G_{\Omega }(\varphi (\cdot , c))(x)=\int _{\Omega }G_{\Omega }(x,y)\varphi (y , c)\ dy .$$

Bounded solutions to \eqref{Main-problem} for a second order elliptic operator $L$ with smooth coefficients were described in \cite{Ghardallou2}. 
Suppose that $L1\leq 0$ and 
$\varphi $ satisfies $(H_1), 
(H_2)$, $(H_3)$. Assume further that $\Omega\subset \mathbb{R}^d$ ($d\geq 3$) and 
{\begin{equation}\label{finitepotential}
\mbox{ there are}\ c>0, x\in S \ \mbox{such that}\ G_{\Omega }(\varphi (\cdot , c))(x)<\infty 
 \footnote{Then $G_{\Omega }(\varphi (\cdot , c))$ is finite at any point because
	$\varphi $ is locally in the Kato class.}.
\end{equation}}
 Then positive continuous
solutions of \eqref{Main-problem} bounded by $c$ are in one-to-one correspondence with positive
$L$-harmonic functions bounded by $c$ given by
\begin{equation}\label{oneone}
h=u+G_{\Omega}(\varphi(\cdot,u)), \hbox{ in }\Omega.
\end{equation}
For a bounded regular domain $\Omega $ and { $f\in C(\partial
	\Omega )$, $f\geq 0$} \eqref{oneone} becomes 
\begin{equation}\label{oneoner}
h_{\Omega}=u+G_{\Omega}(\varphi(\cdot,u)), \hbox{ in }\Omega.
\end{equation}
where $h_{\Omega }$ is  the $L$-harmonic function in $\Omega $
having $f$ as its boundary values \footnote{{ In a bounded regular domain
		$(H_1)$ implies that $G_{\Omega}(\varphi(\cdot,c))$ vanishes on $\partial \Omega $}}. We will write $u=U^{\varphi}_{\Omega }f$. 
\eqref{oneone} and \eqref{oneoner} allow us, in what follows, to go beyond $\varphi (x,t)=p(x)t^{\gamma }$.
\subsection{Bounded solutions for $\L $ under $(H_1)-(H_3)$.}
Existence of bounded solutions is characterized as follows: 
\begin{theorem}\label{bounded-solution}
	Suppose that $(H_1)-(H_3)$  are satisfied for $\L $ and $\varphi$ is radial with respect to
	the first variable. Then  \eqref{subproblem} has a nontrivial nonnegative
	bounded solution if and only if
	\begin{equation}\label{condbounded}
	\mbox{there exists}\  c_0>0\ \mbox{such that}\  I\varphi (\cdot , c_0)<\infty . \end{equation}
	Moreover, if \eqref{condbounded} is satisfied then there is a radial bounded solution.
\end{theorem}
\begin{remark}
	For $\L =\Delta $,  $\varphi (x,t)=p(x)t^{\gamma }$ and $p\in L^{\infty}_{loc}(\mathbb{R}^d)$ {(locally in $L^{\infty }(\mathbb{R}^d)$)} Theorem
	\ref{bounded-solution} was proved in \cite{ElmabroukHansen}. See also \cite{LairWood2000} for H\"older
	continuous $p$.
\end{remark}
\begin{remark}\label{remark30}
	Let $G$ be the fundamental solution to $\mathcal L$. Notice that $I\varphi (\cdot , c_0)<\infty$ is equivalent to $ G ( \varphi(\cdot,c_0))<\infty$ both
	on Euclidean and harmonic space, which follows from writing the volume element in radial
	coordinates, see \eqref{volume}. Indeed, on $NA$,
	\begin{equation*}
	\begin{aligned}\label{Greenident}
	G(\varphi(\cdot,c_0))(0)  &= \int_{NA} G(0,s )\varphi(s, c_0)\,dm_L(s)\\
	&= \int_0^{+\infty} G(r) \varphi(r,c_0) 2^{p+q}\sinh(\frac{r}{2})^{p+q}\cosh(\frac{r}{2})^q\,dr
	\end{aligned}
	\end{equation*}
	and in view of Theorem \ref{estimation-Green-r-big}, $G(r)\approx e^{-Qr}$ as $r\to \infty $.
	Therefore, the last inequality is finite iff $I\varphi (\cdot ,c_0)<\infty $. For $\R ^d $ we use
	estimates for the fundamental solution and we proceed in the same way.
\end{remark}


\begin{proof}[Proof of Theorem \ref{bounded-solution}]
	Sufficiency of \eqref{condbounded} follows directly from \eqref{oneone} and  Remark \eqref{remark30}. So it remains to prove its necessity. We will do it for $NA$. For the Euclidean space it goes in the same way: one has to use the formula for the fundamental solution there.
	
	Suppose that \eqref{subproblem} has a nontrivial nonnegative bounded solution $w$ and let
	\begin{equation}\label{ball}
	B_n=\{s\in NA: d(s,0)<n\}. \end{equation} By Lemma \ref{radiality-in-bounded}, 
	for every $c>0$, $U_{B_n}^{\varphi}c$ is a
	radial solution of \eqref{subproblem} in $B_n$. Also if $\displaystyle {c\geq \sup_{NA} w}$, then
	$u=\lim\limits_{n\to +\infty} U_{B_n}^{\varphi}c$ is a nontrivial  solution. Moreover, $u$ is $\L$-subharmonic
	radial in $NA$, hence, by the maximum principle for elliptic operators, it follows that 
	\begin{equation}\label{subharmonic}
	u(x)\leq u(x_0),\quad \mbox{if}\ \ d(x,0)\leq d(x_0,0).\end{equation}
	We fix $r_0$ such that $u(r_0)>0$. Then $$0\leq \varphi(r,u(r_0))   \leq \varphi(r,u(r)),\
	\hbox{ for }\ r\geq r_0.$$
	
	In addition, $h_u=\lim\limits_{n\to +\infty} H_{D_n}u$ is a positive bounded $\L$-harmonic function
	such that
	$$h_u=u+G(\varphi(\cdot,u)), \hbox{ in }NA.$$
	In particular,
	\begin{equation*}
	\begin{aligned}
	G(\varphi (\cdot , u))(0)    &= \int_{NA} G(0, s)\varphi(s, u(s))\,dm_L(s)\\
	&= \int_0^{+\infty} G(r) \varphi(r,u(r)) 2^{p+q}\sinh(\frac{r}{2})^{p+q}\cosh(\frac{r}{2})^q\,dr .
	\end{aligned}
	\end{equation*}
	Consequently, $$ \int_{r_0}^{+\infty} G(r) \varphi(r,u(r_0))
	2^{p+q}\sinh(\frac{r}{2})^{p+q}\cosh(\frac{r}{2})^q\,dr<\infty .$$ And by the estimate $G(r)\approx
	e^{-Qr}$ for large $r$ and ($H_1$), the conclusion follows with $c_0=u(r_0)$.

	
	
	
	

\end{proof}

\subsection{Large solutions under $(H_1)-(H_4)$}
Now we assume that $\varphi$ is concave with respect to the
second variable and we characterize existence of large solutions. The theorem is as follows

\begin{theorem}\label{large-solution-in-radial-case}
	Suppose that $(H_1)-(H_4)$ are satisfied and $\varphi$ is radial with respect to the first variable. Then  there exists a large solution to \eqref{subproblem}
	if and only if 
	\begin{equation}\label{condlarge}
	\mbox{ for every} \ c>0, \ I\varphi (\cdot , c)=\infty .
	\end{equation}
	Moreover, if \eqref{condlarge} is satisfied then there is a radial large solution.
\end{theorem}
\begin{remark}
	As before, $I\varphi(\cdot ,c)=\infty$ is equivalent to \newline $ G (
	\varphi(\cdot,c))(0) =+\infty$, which implies that $G\varphi (\cdot , c)$ is identically infinity.
\end{remark}
\begin{remark}
	For $\L =\Delta $, $\varphi (x,t)=p(x)t^{\gamma }$ and $p\in L^{\infty }_{loc}(\mathbb{R}^d)$
	Theorem \ref{large-solution-in-radial-case} was proved in \cite{ElmabroukHansen}. See also
	\cite{LairWood2000} for continuous $p$. As before, the results of \cite{Ghardallou2} and \cite{D-G} allow us, in what follows, to go beyond $\varphi (x,t)=p(x)t^{\gamma }$.
	
\end{remark}

The main ingredients of the proof of Theorem \ref{large-solution-in-radial-case}
are Theorem \ref{existence1}, below, see \cite{D-G} and Lemma \ref{the-function-z}. For the latter we need a  Harnack type inequality which is proved in the last section. Both Theorem \ref{existence1} and Lemma \ref{the-function-z} hold for general elliptic operators.


\begin{theorem}(see \cite{D-G})\label{existence1}
	Let $L$ be a second order elliptic operator with smooth coefficients defined on a Green domain
	$\Omega \subset \mathbb{R}^d$, $d\geq 3$, satisfying $L1=0$. Suppose that $(H_1)-(H_4)$ hold and
	that there is a bounded solution to
	$$
	L u(x) - \varphi (x, u(x))=0.
	$$
	Then there is no large solution.
\end{theorem}

\begin{proof}[Proof of Theorem \ref{large-solution-in-radial-case}]
	Necessity of \ref{condlarge} follows from Theorem \ref{bounded-solution}. In fact, by Theorem
	\ref{bounded-solution}, if there exists $c_0>0$ such that $I\varphi(\cdot ,c_0) < \infty$ then
	\eqref{subproblem} has a nontrivial nonnegative bounded solution in $S$. We conclude by the
	Theorem \ref{existence1}, that there is no large solution of \eqref{subproblem}.
	
	Now let us focus on sufficiency. We assume $I\varphi(\cdot ,c) =\infty$ for
	every $c>0$ and we prove that for every $\alpha>0$ there exists $u$ nonnegative radial solution of
	\eqref{subproblem} such that $u(0)=\alpha$ and $\lim\limits_{d(s)\to\infty}u(s)=+\infty.$
	
	Let $z_n$ be defined as in Lemma \ref{the-function-z} for $B_n$, see \eqref{ball}. So there exists $\nu>0$ such that
	$z_n(\nu)=\alpha$.
	Let $\lambda_n=\inf \{\lambda>0, z_n(\lambda)=\alpha\}$ and $u_n=U_{B_n}^\varphi\lambda_n$. As a
	result, for every $n\in\mathbb{N}$, $u_n(0)=\alpha=u_{n+1}(0)$. Also $u_n, u_{n+1}$ are both
	continuous radial solutions in $B_n.$ By Corollary \ref{coincide-from-origin-to-all-domain},
	$u_n=u_{n+1}$ in $B_n$. We conclude that $u=\lim\limits_{n\to +\infty } u_n$ is also a radial
	solution and $u(0)=\alpha.$
	
	Now it is remains to prove that $\lim\limits_{d(s)\to\infty}u(s)=+\infty.$ According to Theorem
	\ref{bounded-solution}, if for every $c>0$, $I\varphi(\cdot ,c) =\infty$ then there is
	no ``bounded nontrivial solution''. Though, $u$ is nontrivial, so $u$ is unbounded. Also, $u$ is
	radial $\L$-subharmonic, hence by \eqref{subharmonic},
	$\lim\limits_{d(s)\to\infty}u(s)=+\infty.$
\end{proof}

\subsection{Useful Lemmas}
This section contains a few lemma used already in the proof of Theorems \ref{bounded-solution} and \ref{large-solution-in-radial-case}. We begin with a comparison principle. For
$L=\Delta $ it was proved in \cite{ELMabrouk2004}, the general case is similar, see
\cite{Ghardallou2}. { $\mathcal{L}^1_{loc}(\Omega)$ denotes the space of functions locally integrable on $\Omega $.}
\begin{lemma}\label{comparaison-semi-elliptic}
	\vspace{0.2cm} Let $L$ be an elliptic operator with smooth coefficients defined in $\Omega $. Suppose that $\varphi $ satisfies ($H_2$). Let $u,v \in \mathcal{C}(\Omega)$ such
	that $Lu,Lv\in{\mathcal{L}}^1_{loc}(\Omega)$. If $$
	L u-\varphi(\cdot,u)\leq L v-\varphi(\cdot,v), \hbox{in $\Omega$},$$
	in the sense of distributions and
	$$ \liminf\limits_{\underset{y\in\partial \Omega}{x\to y}}{(u-v)(x)}\geq0.
	$$
	Then:$$u-v\geq0\,\,in\,\,\Omega.$$
	
\end{lemma}

\begin{lemma}
	\label{the-function-z}
	Let $L$ be an elliptic operator with smooth coefficients defined in a
	bounded regular domain $D$, $x_0\in D$ and suppose that $\varphi$ satisfies $(H_1)-(H_3)$. For $x_0\in D$, $\lambda \in \R _+$ we define 
	$$
	z(\lambda )=U_D^{\varphi }\lambda (x_0).$$
	Then
	\begin{itemize}
		\item[(a)] $z(0)=0.$
		
		\item[(b)] For every $\lambda\geq \nu\geq 0$, $$0\leq z(\lambda)-z(\nu)\leq \lambda-\nu.$$

		\item[(c)] {We suppose in addition that $\varphi$ satisfies $(H_4)$ (resp. $(H'_1)$) then }
		$\lim\limits_{\lambda\to+\infty}z(\lambda)=+\infty.$
	\end{itemize}
\end{lemma}

\begin{proof}\ \\
	
	
	(a) If $\lambda=0$ then $U_D^\varphi\lambda=0$, consequently $z(0)=0$.
	
	\medskip 
	(b) By Lemma \ref{comparaison-semi-elliptic}, if $\lambda\geq \nu\geq 0$ then $U_D^\varphi \lambda
	\geq U_D^\varphi \nu$ and hence $$0\leq z(\lambda)-z(\nu).$$
	In addition,
	\begin{eqnarray*}
		z(\lambda)-z(\nu) &=&  U_D^\varphi \lambda (x_0)-U_D^\varphi\nu (x_0) \\
		&=& (\lambda-G_D(\varphi(\cdot, U_D^\varphi \lambda ))(x_0))-(\nu-G_D(\varphi(\cdot, U_D^\varphi \nu ))(x_0)) \\
		&=& (\lambda-\nu)+G_D(\varphi(\cdot,U_D^\varphi\nu)(x_0)-\varphi(\cdot,U_D^\varphi\lambda))(x_0)
		\leq \lambda -\nu .
	\end{eqnarray*}
	because $\varphi $ is increasing with respect to the second variable and so
	$\varphi(\cdot, U_D^\varphi \nu )-\varphi(\cdot, U_D^\varphi \lambda) \leq 0$.
	
	\medskip
	(c) Suppose first that $\varphi $ satisfies $(H_4)$. Let $\lambda\geq 1$ and $y\in D$ such that $U_D^\varphi 1 (y)\not=0$. Then by  Theorem
	\ref{Mainthe-Harnack-ineq} there exists a constant $c>0$ independent
	on $\lambda$ such that
	$$ U_D^{\varphi} \lambda(y)\leq c( U_D^{\varphi} \lambda(x_0)+1).$$
	Also, by concavity and Lemma \ref{comparaison-semi-elliptic}: $$\lambda U_D^{\varphi} 1(y)\leq U_D^{\varphi} \lambda(y).$$
	Consequently $$\lambda U_D^{\varphi} 1(y) \leq c( z(\lambda)+1).$$
	Therefore, $$\lim_{\lambda\to +\infty}z(\lambda)=+\infty.$$
	
	{Suppose now that $\varphi$ satisfies $(H'_1)$}, by Theorem 4 in \cite{D-G} there is $\varphi _1$ such that $\varphi(x,t) \leq \varphi_1(x,t)$ and $\varphi _1$ satisfies ($H_1$)-($H_4$). Let $\lambda\geq 0$ and $$
	z_1(\lambda )=U_D^{\varphi_1 }\lambda (x_0), \;\; u =U_D^{\varphi }\lambda, \mbox{  et  } v=U_D^{\varphi_1 }\lambda   .$$
	Since $Lv - \varphi(\cdot, v )\geq 0$ then by Lemma \ref{comparaison-semi-elliptic} $$v\leq u,\mbox{  in  } D.$$
	Consequently $$z_1(\lambda) \leq z(\lambda).$$
	and we conclude that $$\lim_{\lambda\to+\infty}z(\lambda)=+\infty.$$

\end{proof}

We need also a lemma that gives comparison between radial sub-solutions and super-solutions to \eqref{subproblem} in
\begin{equation}\label{ballR}
B_R=\{ x\in S: d(x,0)<R\}.\end{equation}
We follow arguments from \cite{ElmabroukHansen}.
\begin{lemma}\label{comparaison-under-radiality}
	\vspace{0.2cm} Assume that $\varphi $ satisfies ($H_1$)-($H_4$). Let $u,v\in \mathcal{C}(B_R)$ be
	radially symmetric such that $\L u,\L v\in{\mathcal{L}}^1_{loc}(B_R)$ and $v>0$ in
	$B_R\backslash \{0\}$.
	We suppose that
	$$
	\L u-\varphi(\cdot,u)\leq \L v-\varphi(\cdot,v) \hbox{ in }B_R,$$
	and
	$$ \liminf\limits_{x\to 0} \frac{u(x)}{v(x)}\leq 1.$$
	Then:$$u-v\leq0\,\,in\,\,B_R.$$
	
\end{lemma}

\vspace{0.5cm}

\begin{proof}

	Let $0<r<R$. Since $u$ and $v$ are constant on $\partial B_r$ we may denote
	$\alpha=\frac{u(x)}{v(x)}$ if $x\in\partial B_r$. Suppose that $\alpha>1$ and define $w=\alpha v.$
	Using concavity and $\varphi(\cdot,0)=0$,
	\begin{equation*}
	\begin{aligned}
	\L w-\varphi(\cdot,w) &= \alpha \L v-\varphi(\cdot,\alpha w)\\
	&\geq  \alpha \L v-\alpha \varphi(\cdot,v)\\
	&\geq \L v-\varphi(\cdot,v)
	\end{aligned}
	\end{equation*}
	In addition $w=u$ on $\partial B_r$. Consequently, by Lemma \ref{comparaison-semi-elliptic} ,
	$w\leq u$ in $B_r$. As a result, $1<\alpha\leq \frac{u(x)}{v(x)}$ on $B_R\backslash \{ 0\}$.
	Though, $\liminf\limits_{x\to 0} \frac{u(x)}{v(x)}\leq 1$ which gives a contradiction.
	Finally, we conclude that $\alpha \leq 1$ and $u\leq v$ in $B_R$.

\end{proof}

Consequently, we can deduce the following corollary:

\begin{corollary}\label{coincide-from-origin-to-all-domain}
	
	\vspace{0.2cm} Suppose that the assumptions of Lemma \ref{comparaison-under-radiality} are
	satisfied in $\mathbb{R}^d$ on in a ball $B_R$ only. Then
	If $u, v$ are radially symmetric both solutions to \eqref{subproblem} in
	$S$, $u(0)=v(0)$ then $u=v$ in $S$.

\end{corollary}

\subsection{Large solutions without concavity}
First we notice that condition \eqref{condlarge} remains sufficient for existence of large solutions without assumption ($H_4$). Indeed, we have

\begin{theorem}
	\label{Large-solution-without-concavity}
	Let $\varphi(\cdot,t)\in \mathcal{L}^{\infty}_{loc}(S)$ for every $t>0$ radially symmetric with respect to the first variable satisfying $(H'_1)$, $(H_2)$,
	$(H_3)$ \footnote{ See introduction for $(H'_1)$.}. 
	If for every $c>0$, $I\varphi (\cdot , c)=\infty $
	then there exists a radial large solution of \eqref{subproblem}.
\end{theorem}
{We intend to proceed as in the proof of Theorem \ref{large-solution-in-radial-case} but it is not so easy to constuct a global not trivial solution as the limit of solutions $u_n$ in $B_n$. However, given a sequence of solutions of \eqref{subproblem} which coincide at one point, we prove that there exists a subsequence which converges uniformly to a solution of \eqref{subproblem}. }

\begin{lemma}\label{subsequence-Ascoli-converge-uniformly}
	Let $\varphi(\cdot,c) \in \mathcal{L}^{\infty}_{loc}(S)$ for every $c>0$ and suppose that
	$(H'_1)$, $(H_2)$ and $(H_3)$ are satisfied. Given $\alpha
	>0$ and $x_0\in S$ let $u_n$ be a solution of \eqref{subproblem} in $B_n$
	satisfying $u_n(x_0)=\alpha$.
	Then there exists a subsequence $(u_{ n_k })$ that converges uniformly to a solution
	$u$ of \eqref{subproblem} in $S$ and satisfying $u(x_0)=\alpha$.
\end{lemma}	
\begin{proof}
	We will use Arzela-Ascoli theorem to prove the existence of such a sequence.
	Let $D$ be a regular bounded domain and let $n_0\in\mathbb{N}$ be such that $\overline{D}\subset
	B_{n_0}$ and $x_0\in D$. We begin by equicontinuity of $\{u_n, n> n_0\}$ in $\overline{D}$. It is
	enough to prove that there exists $C>0$ such that for every $n> n_0$ and $i\in\{1,\dots, d\}$
	\begin{equation}\label{derivbound}
	\Big |   \sup_{  x \in \bar D}\frac{\partial}{\partial x_i} u_n(x)\Big | \leq
	C.\end{equation}
	
	We have $$  u_n  +  G_{B_{n_0}}(\varphi(\cdot,u_n))=H_{B_{n_0}} u_n=h'_n, \hbox{ in }B_{n_0}. $$
	Moreover, in $B_{n_0}$ we have $G_{B_{n_0}}(\varphi(\cdot,u_n))=G(\varphi(\cdot,u^r_n))-g_n$, where $g_n$ is an $\L$-harmonic nonnegative function, $u^r_n$ is $u_n$ restricted to $B_{n_0}$ and
	$$
	G(\varphi(\cdot,u^r_n))= \int _{B_{n_0}}G(x,y)\varphi (y ,u_n(y))\ dm_L(y).
	$$
	Therefore, 
	$$
	u_n+G(\varphi(\cdot,u^r_n))=h'_n+g_n=h_n.$$ 
	By Harnack inequality for $\L$ there is 
	$\beta>0$ such that for every $x\in\overline{D}$ and $n> n_0$
	$$\Big |\frac{\partial}{\partial x_i} h_n(x)\Big | \leq \beta
	\inf_{y\in \bar D} h_n(y).$$
	Also $$0\leq h_n(x)   \leq  \sup_{y\in B_{n_0}} u_n(y).$$
	Furthermore, by Theorem \ref{Mainthe-Harnack-ineq} there exists $\gamma>0$
	\begin{equation}\label{Harnack-solutions} \sup_{x\in B_{n_0}} u_n(x) \leq \gamma
	(1+u_n(x_0))=\gamma(1+\alpha), \quad \mbox{for}\ n>n_0.
	\end{equation}
	Hence, $$\sup_{x\in \bar D}\Big | \frac{\partial}{\partial x_i}h_n(x)\Big |  \leq  \beta \gamma
	(1+\alpha) \hbox{ for  }\ n> n_0.$$ Secondly, we are going to prove that there is $C_1>0$ such that
	\begin{equation}\label{Greenbound}
	\Big | \sup_{x\in \bar D} \frac{\partial}{\partial x_i} G (\varphi(\cdot,u^r_n))(x)  \Big | \leq
	C_1\end{equation} for every $n> n_0$. In view of \eqref{gradestim},
	\begin{equation}
	\left| \frac{\partial}{\partial x_i} G(x,y)\right |\leq C|x-y|^{-d+1}.
	\end{equation} 
	Indeed, in a bounded domain we may replace left-invariant derivatives by the partial ones and the Riemannian distance by the euclidean distance using the identification of harmonic $NA$ with $\R ^d$ described in Section \ref{NAgroups}. We have the same estimate for $G$ corresponding to the Laplace operator in $\R ^d$. 
	Observe, that \eqref{Harnack-solutions} together with
	$\varphi(\cdot,c)\in\mathcal{L}^{\infty}_{loc}(S)$ implies $\varphi(\cdot,u_n)$ is uniformly
	bounded in $B_{n_0}$. Hence  
	\begin{equation}\label{derivint}
	\frac{\partial}{\partial x_i} G (\varphi(\cdot,u^r_n))(x) = \int_{B_{n_0}}  \frac{\partial G (x,y)}{\partial x_i}    \varphi(y,u_n(y))  \,{ dm_L(y)}, \hbox{ for }x\in {B_{n_0}}.\end{equation}
	
	Therefore, $$\Big | \frac{\partial}{\partial x_i} G (\varphi(\cdot,u^r_n))(x)  \Big | \leq
	\sup_{x\in B_{n_0}} \varphi(x,\gamma (1+\alpha )) \int_{B_{n_0}} \Big | \frac{\partial}{\partial
		x_i}  G (x,y)  \Big | \,{dm_L(y)}$$ which is bounded independently of $n$.
	
	Consequently, \eqref{Greenbound} follows and we conclude \eqref{derivbound}.
	Therefore, the family of functions $\{  u_n ,  n > n_0\}$ is equicontinuous uniformly bounded on
	$\overline{D}$. By Arzela-Ascoli theorem there exists a subsequence $(u_{n_k})$ that converges
	uniformly on $\overline{D}$. Hence, by Lemma 8 in \cite{Ghardallou2}, $v_0= \lim\limits
	_{n\to +\infty} u_{n_k }$ is a solution of \eqref{Main-problem} in $D$
	satisfying $v_0(x_0)= \alpha$.
	
	The argument presented above applies to $D=B_n$ for any $n$. So we start with $D=D_1$ and we
	construct a subsequence $(u_{n_k})$ converging uniformly on $D_1$ to a solution of
	\eqref{Main-problem}. Then we pass to $D=D_2$ and we select out of
	$(u_{n_k})$ an analogous subsequence for $D_2$. We conclude by the diagonal method.
\end{proof}

\begin{proof}[Proof of Theorem \ref{Large-solution-without-concavity}]
	Let $\alpha>0$.	We prove that there exists $u$ nonnegative radial solution of
	\eqref{subproblem} such that $u(0)=\alpha$ and $\lim\limits_{d(s)\to\infty}u(s)=+\infty.$
	Let $z_n$ be defined as in Lemma \ref{the-function-z} for $B_n$, see \eqref{ball}. So there exists $\nu>0$ such that
	$z_n(\nu)=\alpha$.
	Let $\lambda_n=\inf \{\lambda>0, z_n(\lambda)=\alpha\}$ and $u_n=U_{B_n}^\varphi\lambda_n$. As a
	result, for every $n\in\mathbb{N}$, $u_n(0)=\alpha=u_{n+1}(0)$. Also $u_n, u_{n+1}$ are both
	continuous radial solutions in $B_n.$ By Lemma \ref{subsequence-Ascoli-converge-uniformly}, there exists a subsequence $(u_{ n_k })$ that converges uniformly to a radial solution
	$u$ of \eqref{subproblem} satisfying $u(0)=\alpha$.
	
	Now it remains to prove that $\lim\limits_{d(s)\to\infty}u(s)=\infty.$ In view of Theorem
	\ref{bounded-solution}, if for every $c>0$, $I\varphi (\cdot , c)=\infty$ then there is no ``bounded nontrivial solution" in $S$. Though, $u$ is nontrivial, so $u$ is
	unbounded in $S$. Also, $u$ is radial 
	$\L$-subharmonic, so $u$ is increasing, consequently
	$\lim\limits_{d(s)\to\infty}u(s)=\infty.$
	
\end{proof}
For the full characterization we need a little bit more about $\varphi $.

\begin{theorem}\label{largenoconcavity}
	Suppose that $\varphi (s,t)=p(s)\psi (t)$, satisfies $(H_2)$,$(H_3)$ and  $p\in  \mathcal{L}^{\infty}_{loc}(S)$ is radial and $\psi (t)\leq c(t+1)$. Then 
	there is a large solution to \eqref{subproblem}
	if and only if $I(p)=\infty$. Moreover, if the latter is satisfied there is a radial large solution. 
\end{theorem}

In this case we make use of the following analogue of Theorem \ref{existence1},
see \cite{D-G}.
\begin{theorem}\label{existence}
	Let $L$ be a second order elliptic operator with smooth coefficients defined on a domain $\Omega
	\subset \mathbb{R}^d$, $d\geq 3$, satisfying $L1=0$. Suppose that $\varphi (x,t)=p(x)\psi (t)$,
	$(H'_1),(H_2),(H_3)$ hold and that there is a bounded solution to \eqref{subproblem}. Then there is no
	large solution.
\end{theorem}

\begin{proof}[Proof of Theorem \ref{largenoconcavity}]
	Sufficiency of $I(p)=\infty $ is guaranteed by Theorem \ref{Large-solution-without-concavity}. It
	is enough now to prove that existence of a large solution implies $I(p)=\infty$. Suppose that
	this is not true i.e there is a large solution and $I(p)<\infty $. 
	Hence there is a nontrivial nonnegative bounded solution of equation \eqref{subproblem}. But then, by Theorem \ref{existence} there is no large solution for \eqref{subproblem}. This gives
	a contradiction.
\end{proof}

\section{$NA$ groups and the fundamental solution for $\L $.} \label{NAgroups}

\subsection{Basic structure of $S=NA$}\label{basic-structure-NA}
In this section, we recall briefly the structure of $NA$ groups. For more details we refer to
\cite{Anker-Damek-Yacoub,BTV,D-R,Rouvière}.

Let $\mathfrak{n}$ be a two step nilpotent Lie algebra, equipped with an inner product $\langle
\cdot, \cdot \rangle$. Denote by $\mathfrak{z}$ the center of $\mathfrak{n}$ i.e. $\mathfrak{z}=\{
a\in \mathfrak{n}, \forall x\in \mathfrak{n}~~ [a,x]=0_{\mathfrak{n}} \}$ and $\mathfrak{v}$ be the
orthogonal complement of $\mathfrak{z}$ in $\mathfrak{n}$. Consequently,
$[\mathfrak{n},\mathfrak{z}]=0_{\mathfrak{n}}$, $[\mathfrak{n}, \mathfrak{n}]\subset \mathfrak{z}$
and $\mathfrak{n}=\mathfrak{z}\bigoplus\mathfrak{v} $.

Additionally we suppose that for every $z\in \mathfrak{z}$ there is a linear map \newline
$J_z:\mathfrak{v}\to\mathfrak{v}$ satisfying
\begin{equation}\label{31}
\langle J_z (v), v'\rangle=\langle z,[v,v']\rangle, \hbox{    for every  }v,v'\in \mathfrak{v},
\end{equation}
and
\begin{equation}\label{32}
J_z\circ J_z(v)=-|z|^2v=-\langle z,z\rangle v, \hbox{  for every }v\in \mathfrak{v}.
\end{equation}
Then $\mathfrak{n}$ is called a \emph{Lie algebra of Heisenberg type}.
We denote $$p=\dim \mathfrak{v},~~~~q=\dim \mathfrak{z}.$$

The corresponding Lie group $N$ is called a \emph{Heisenberg Lie group}. Via the exponential map,
we shall identify $\mathfrak{n}$ and $N$, hence the multiplication in $N\equiv
\mathfrak{n}=\mathfrak{z}\bigoplus\mathfrak{v} $ reads $$x.x'=x+x'+\frac{1}{2}[x,x'].$$

We consider the semi-direct product $S=NA=N\ltimes\mathbb{R^*_+}$ defined by
\begin{equation*}
(v,z,a).(v',z',a')=(v+a^\frac{1}{2}v',z+az'+\frac{1}{2}a^\frac{1}{2}[v,v'],aa'),
\end{equation*}
for every $v,v'\in \mathfrak{v}$, $z,z'\in \mathfrak{z}$, $a,a'\in \mathbb{R}_+^*$. 
$NA$ is a solvable Lie group, with the corresponding Lie algebra
$\mathfrak{s}=\mathfrak{v}\oplus \mathfrak{z}\oplus \mathbb{R}$.
Let $d$ be the left invariant Riemannian distance on $NA$ induced by the inner product
$$\langle (v,z,l),(v',z',l')\rangle_{0_s}=\langle v,v'\rangle+\langle
z,z'\rangle+ll',$$ defined originally on $T_{0_s} (NA)=\mathfrak{s}$\footnote{  $\; \; 0_s$ is the neutral element of $S$.}.
Then the corresponding Riemannian volume is given by
$$dm_L=a^{-Q}dXdZ\frac{da}{a},$$
where $Q=\frac{p}{2}+q$ is the homogeneous dimension on $N$. $dm_L$ is at the same time a left invariant Haar measure on $NA$.

$NA$ has global geodesic coordinates i.e. for $s\in NA$ we may write $s=(r,\theta)$, where $r=d(s,0_s)$, $\theta $ belongs to the unit sphere in $\mathfrak{s}\equiv \R^{p+q+1}$ and the above decomposition is a diffeomorphism.
The Riemannian volume in geodesic coordinates $s=(r,\theta)$ is given by
\begin{equation}\label{volume}
dm_L(s)=2^{p+q} (\sinh(\frac{r}{2}))^{p+q})(\cosh(\frac{r}{2}))^q dr d\rho_0(\theta),\end{equation} where
$d\rho_0(\theta)$ is the surface measure on the unit sphere in $\mathbb{R}^{p+q+1}$.

Let $\L$ be the Laplace-Beltrami-operator in NA. It is symmetric with respect to $dm_L$ and it can be written as
$$ \L = L_0 + L_1,$$ where
\begin{equation}\label{expression-of-L}
\begin{aligned}
L_0 &= (1-Q) a\partial_a + a^2 \partial_a^2 + a( a + \frac{1}{4} |X|^2 ) \Delta_Z + a \Delta_X\\
L_1 &= a^2 \sum_{m+1}^{m+k} \sum_{i=1}^m <[X,e_i] , e_r > \partial_r \partial_i.
\end{aligned}
\end{equation}
$\Delta_X$, $\Delta_Z$ are the Laplace operators on $\mathfrak{v}$ and $\mathfrak{z}$ respectively, $e_1,\dots , e_m$ is an orthonormal basis of $\mathfrak{v}$ and $e_{m+1},\dots e_{m+k}$ is an orthonormal basis of $\mathfrak{z}$.


If a function $f$ depends only on $|X|$, $Z$ and $a$ then $L_1f=0$ Hence, the radial part of $L$ is
given by
\begin{eqnarray*}
	L_{rad} &=&  \partial_r^2+(~~~~\frac{p+q}{2}\coth(\frac{r}{2})~~+~~\frac{q}{2}\tanh(\frac{r}{2})~~~~)\partial_r \\
	&=&\partial_r^2+(~~~~\frac{p}{2}\coth(\frac{r}{2})~~+~~q\coth(r)~~~~)\partial_r .
\end{eqnarray*}
\subsection{Estimates of the fundamental solution}

Let $h_t(s,s_1)$ be the heat kernel for $\L $. Then
$$ h_t(s,s_1) = h_t ({s_1}^{-1} s, 0_s ), $$
$$
h_t ( s,0_s)=h_t ( d(s,0_s))$$
is radial and we have the following estimate proved in \cite{Anker-Damek-Yacoub}, section 5. 
\begin{theorem}\label{estimation-heat-kernel}
	Let $r=r(s)=d(0_s,s)$. Then
	$$h_t(r)\asymp t^{-\frac{3}{2}}(1+r)(1+\frac{1+r}{t})^{\frac{n-3}{2}}\exp(-\frac{Q^2}{4}t-\frac{Qr}{2}-\frac{r^2}{4t}), ~~~~ r\geq 0, t>0.$$
	
\end{theorem}

Theorem \ref{estimation-heat-kernel} will be now used to prove precise estimates of the fundamental solution 
$$G(r)=
\int_0^{+\infty} h_t(r) \,dt$$
for $\L$.
\begin{theorem}\label{estimation-Green-r-big}
	
	There are $c>0$ and $d>0$ such that 
	$$ c\exp(-Qr)\leq G(r)\leq d \exp(-Qr),  \hbox{  }r\geq 1.$$
\end{theorem}

\begin{proof}
	
	We have  $$1\leq 1+\frac{1+r}{t}, ~~ r\geq 0, t>0.$$ Then by Theorem \ref{estimation-heat-kernel},
	there exists a constant $c_1>0$ such that
	$$c_1 (1+r)\exp(-\frac{Qr}{2}) \int_0^{+\infty}t^{-\frac{3}{2}}\exp(-\frac{Q^2}{4}t-\frac{r^2}{4t})\,dt\leq G(r), ~~r> 0.$$
	In addition, according to \cite{selmi_these} p.97:
	$$\int_0^{+\infty}t^{-\frac{3}{2}}\exp(-\frac{Q^2}{4}t-\frac{r^2}{4t})\,dt=(4\pi)^{\frac{3}{2}}g_{3, \frac{Q^2}{4}}(r),$$
	where $g_{n,\lambda}$ is the Green function of $\Delta-\lambda$ in $\mathbb{R}^n$.
	Moreover,\begin{eqnarray*}
		g_{3, \frac{Q^2}{4}}(r) &=&  \frac{Q^2}{4}^\frac{1}{2}g_{3,1}((\frac{Q^2}{4})^\frac{1}{2}r) )\\
		&=& \frac{Q}{2} \frac{1}{4\pi} \frac{1}{r}\exp(-\frac{Qr}{2})
	\end{eqnarray*}
	As a result, \begin{equation*}
	\int_0^{+\infty}t^{-\frac{3}{2}}\exp(-\frac{Q^2}{4}t-\frac{r^2}{4t})\,dt
	= (4\pi)^{\frac{1}{2}}\frac{Q}{2r}\exp(-\frac{Qr}{2}).
	\end{equation*}
	Consequently, there exists a constant $c>0$ such that $$c \exp(-Qr)\leq G(r), ~~r>0.$$ Let us now
	focus on the upper estimate. As above, we use the estimate of Theorem \ref{estimation-heat-kernel}.
	We divide the integral into two parts: $]1+r, +\infty[$ and $]0,1+r[$. For the first interval,
	\begin{eqnarray*}
		t\geq 1+r &\Rightarrow&  1\geq\frac{1+r}{t}\\
		&\Rightarrow& 2\geq \frac{1+r}{t}+1\\
		&\Rightarrow& 2^\frac{n-3}{2}\geq ( \frac{1+r}{t}+1 )^\frac{n-3}{2}.
	\end{eqnarray*}
	Then there exists $d_1>0$ such that $$\int_{r+1}^{+\infty}h_t(r)\,dt\leq d_1 (1+r)
	\exp(-\frac{Qr}{2}) \int_0^{+\infty}t^{-\frac{3}{2}}\exp(-\frac{Q^2}{4}t-\frac{r^2}{4t})\,dt,~~~~r>
	0.$$ Though as before:
	$$\int_0^{+\infty}t^{-\frac{3}{2}}\exp(-\frac{Q^2}{4}t-\frac{r^2}{4t})\,dt = (4\pi)^{\frac{1}{2}}\frac{Q}{2r}\exp(-\frac{Qr}{2}),~~ r>0.$$
	As a result, there exists $d_2>0$ such that
	$$\int_{r+1}^{+\infty}h_t(r)\,dt\leq d_2 \frac{1+r}{r} \exp(-Qr), ~~~r>0.$$
	Furthermore, $r\mapsto \frac{1+r}{r}$ is bounded on $[1,+\infty[$, which implies that
	$$\int_{r+1}^{+\infty}h_t(r)\,dt\leq d_3 \exp(-Qr), ~~~r\geq 1.$$
	Secondly, \begin{eqnarray*}
		t\leq 1+r &\Rightarrow&  1\leq\frac{1+r}{t}\\
		&\Rightarrow& 1+\frac{1+r}{t}\leq 2 \frac{1+r}{t}\\
		&\Rightarrow& (1+\frac{1+r}{t})^\frac{n-3}{2}\leq 2^\frac{n-3}{2} (\frac{1+r}{t})^\frac{n-3}{2}.
	\end{eqnarray*}
	Hence, there exists $d_4>0$ such that, \begin{eqnarray*}
		\int_{0}^{1+r}h_t(r)\,dt &\leq &  d_4 (1+r)^\frac{n-1}{2} \exp(-\frac{Qr}{2}) \int_0^{1+r}t^{-\frac{n}{2}}\exp(-\frac{Q^2}{4}t-\frac{r^2}{4t})\,dt , ~~~r>0\\
		&\leq& d_4 (1+r)^\frac{n-1}{2} \exp(-\frac{Qr}{2}) \int_0^{+ \infty}t^{-\frac{n}{2}}\exp(-\frac{Q^2}{4}t-\frac{r^2}{4t})\,dt , ~~~r>0.
	\end{eqnarray*}
	Again according to \cite{selmi_these} p.97,
	\begin{eqnarray*}
		\int_0^{+ \infty}t^{-\frac{n}{2}}\exp(-\frac{Q^2}{4}t-\frac{r^2}{4t})\,dt &=& (4\pi)^{\frac{n}{2}}g_{n, \frac{Q^2}{4}}(r)\\
		&=& (4\pi)^{\frac{n}{2}} (\frac{Q^2}{4})^\frac{n-2}{2}g_{n, 1}(\frac{Q}{2}r)\\
		&\underset{r\to+\infty}{\sim}& (4\pi)^{\frac{n}{2}} (\frac{Q^2}{4})^\frac{n-2}{2} \frac{1}{2 (2\pi)^{\frac{n-1}{2}}}\frac{\exp(-\frac{Qr}{2})}{(\frac{Qr}{2})^\frac{n-1}{2}}.
	\end{eqnarray*}
	Using that $r\mapsto \frac{1+r}{r}$ is bounded on $[1,+\infty[$, we get that there is a constant
	$d>0$ such that $$ G(r)\leq d \exp(-Qr),~~~~r\geq 1.$$
	
\end{proof}

\begin{theorem}\label{estimation-Green-small-r}
	For $0 < r\leq 1$ we have the following estimate for $G$ and the left-invariant gradient of $\nabla G$:
	\begin{align}
	G(r)\asymp r^{2-n}, \label{Gestim}\\
	\nabla G(r)\leq Cr^{1-n}. \label{gradestim}
	\end{align}
\end{theorem}

\begin{proof}
	
	We have $$\frac{1+r}{t}\leq 1+\frac{1+r}{t}, ~~ r\geq 0, t>0,$$ then by Theorem
	\ref{estimation-heat-kernel} there exists a constant $c_1>0$ such that
	\begin{equation*}
	c_1 (1+r)\int_0^{+\infty}t^{-\frac{3}{2}} (\frac{1+r}{t})^{\frac{n-3}{2}}\exp(-\frac{Q^2}{4}t-\frac{Qr}{2}-\frac{r^2}{4t})\,dt\leq G(r), ~~r> 0.
	\end{equation*}
	Furthermore, $$1\leq 1+r,\hbox{    } r>0,$$ and $$\exp(-\frac{Q}{2})\leq \exp(-\frac{Qr}{2}),\hbox{
	} 0 <r\leq 1.$$ So there exists a constant $c_2>0$ such that
	$$c_2 \int_0^{+ \infty}t^{-\frac{n}{2}}\exp(-\frac{Q^2}{4}t-\frac{r^2}{4t})\,dt\leq G(r), \hbox{ } 0<r\leq 1.$$
	Though, according to \cite{selmi_these}, and using same notation as in the preceding theorem we
	have
	\begin{eqnarray*}
		\int_0^{+ \infty}t^{-\frac{n}{2}}\exp(-\frac{Q^2}{4}t-\frac{r^2}{4t})\,dt &=& (4\pi)^{\frac{n}{2}}g_{n, \frac{Q^2}{4}}(r)\\
		&=& (4\pi)^{\frac{n}{2}} (\frac{Q^2}{4})^\frac{n-2}{2}g_{n, 1}(\frac{Q}{2}r).
	\end{eqnarray*}
	In addition,
	$$g_{n,1}(r)   \underset{r\to 0}{\sim}   \frac{1}{4\pi^{\frac{n+1}{2}}}\gamma(\frac{n}{2}-1)\frac{1}{r^{n-2}}.$$
	So there exists a constant $c>0$ such that
	$$c~~r^{2-n}\leq G(r), \hbox{ }0<r\leq 1.$$
	Now, let us focus on the upper estimates: As above, we use estimates of Theorem
	\ref{estimation-heat-kernel}. We divide the integral into two parts: $]1+r, +\infty[$ and
	$]0,1+r[$. At first, \begin{eqnarray*}
		t\geq 1+r &\Rightarrow&  1\geq\frac{1+r}{t}\\
		&\Rightarrow& 2^\frac{n-3}{2}\geq ( \frac{1+r}{t}+1 )^\frac{n-3}{2}.
	\end{eqnarray*}
	Then there exists $d_1>0$ such that $$\int_{r+1}^{+\infty}h_t(r)\,dt\leq d_1 (1+r)
	\exp(-\frac{Qr}{2}) \int_0^{+\infty}t^{-\frac{3}{2}}\exp(-\frac{Q^2}{4}t-\frac{r^2}{4t})\,dt~~~~r>
	0.$$ Though according to \cite{selmi_these} as before:
	$$\int_0^{+\infty}t^{-\frac{3}{2}}\exp(-\frac{Q^2}{4}t-\frac{r^2}{4t})\,dt = (4\pi)^{\frac{1}{2}}\frac{Q}{2r}\exp(-\frac{Qr}{2}),~~ r>0.$$
	However, for $0<r\leq 1$, $$r^{-1}\leq r^{2-n},$$ so  $$ \int_{r+1}^{+\infty}h_t(r)\,dt\leq d_2
	r^{2-n}$$ for some $d_2>0$. Secondly, \begin{eqnarray*}
		t\leq 1+r &\Rightarrow&  1\leq\frac{1+r}{t}\\
		&\Rightarrow& (1+\frac{1+r}{t})^\frac{n-3}{2}\leq 2^\frac{n-3}{2} (\frac{1+r}{t})^\frac{n-3}{2}.
	\end{eqnarray*}
	Hence, there exists $d_3>0$ such that, \begin{eqnarray*}
		\int_{0}^{1+r}h_t(r)\,dt &\leq &  d_3 ~~\int_0^{1+r}t^{-\frac{n}{2}}\exp(-\frac{Q^2}{4}t-\frac{r^2}{4t})\,dt , ~~~ 1\geq r>0.
	\end{eqnarray*}
	As before, \begin{eqnarray*}
		\int_0^{+ \infty}t^{-\frac{n}{2}}\exp(-\frac{Q^2}{4}t-\frac{r^2}{4t})\,dt
		&=& (4\pi)^{\frac{n}{2}} (\frac{Q^2}{4})^\frac{n-2}{2}g_{n, 1}(\frac{Q}{2}r).
	\end{eqnarray*}
	In addition,
	$$g_{n,1}(r)   \underset{r\to 0}{\sim}   \frac{1}{4\pi^{\frac{n+1}{2}}}\gamma(\frac{n}{2}-1)\frac{1}{r^{n-2}}.$$
	Hence, there exists $d_3>0$ such that, $$\int_{0}^{1+r}h_t(r)\,dt \leq d_3 ~~ r^{2-n}$$ and we get
	{ \eqref{Gestim}. \eqref{gradestim} follows directly from the estimate for the gradient of $h_t$ contained in \cite{Anker-Damek-Yacoub}, Corollary 5.49 and a similar calculation as above.}
\end{proof}


\subsection{The Green kernel}
Let
$$G(s,s_1)= G(s_1^{-1}s ), \quad 
\mbox{for}\quad s, s_1\in NA .$$
Then $G(\cdot, \cdot)$ is the Green function for $\L $ on $NA$ i.e. 
\eqref{delta} and \eqref{potential} are satisfied with respect to the measure $dm_L$.
$$ G(f)=\int_{NA}
G(s,s_1)f(s_1)dm_L(s_1) $$
is the Green operator
i.e 
$$ \L( G (f) ){=} - f,$$
for $f\in \mathcal{C}_c^{\infty}(S)$.
The proof of the above properties is standard provided all the integrals are well defined which follows, in particular, from the estimates included in Theorems \ref{estimation-heat-kernel}, \ref{estimation-Green-r-big} and \ref{estimation-Green-small-r}.

Now we prove that the Green operator preserves radiality as it does on the Euclidean space. Let $G_B$ be the Green
function for $B=B_R$, see \eqref{ballR}
Then
$$
G_B(s,s_1)=G(s,s_1)-\tilde{h}(s,s_1),$$ where for a fixed $s_1$, $\tilde{h}(\cdot , s_1)$ is the $\L$-harmonic function in $B$ with the boundary value $G(\cdot , s_1)$.

\begin{theorem}\label{G_Bf-radial}
	If a continuous function $f$ is radial then $G_Bf$ is
	radial too.
\end{theorem}
\begin{proof}
	It is enough to prove that $$G_B(f)=M(G_Bf),$$ where for a continuous function $u$
	$$M(u)(r)=\frac{1}{\sigma_e(r)}\int_{S_e(r)}u(r,\theta)\,dr d\rho _0 (\theta ) .$$
	Here $d\rho _0$ is the surface measure on the unit sphere in $\mathfrak{s}$ and $\sigma _e(r)$ is the measure of the sphere $S_e(r)$.
	In addition $NA$ is a harmonic
	manifold so $$\L(M (u))=M(\L(u)),$$ see  \cite{Rouvière}. Consequently,
	\begin{equation*}
	\begin{aligned}
	\L(M(G_Bf)) &= M(\L(G_Bf))\\
	&= M(-f)
	\end{aligned}
	\end{equation*}
	But if $f$ is radial then $M(f)=f$. As a result $$\L(G_Bf-M(G_Bf))=0, \hbox{ in }B.$$ Moreover, $G_Bf=0$ that on the boundary $\partial B$ of $B$ and so is
	$M(G_Bf)$ on $\partial B$. We conclude that 
	$$G_Bf=M(G_Bf).$$
\end{proof}

\begin{corollary}
	Let $f$ be a radial continuous function in $NA$ such that $Gf$  is well defined. Then $Gf$ is
	radial too.
	
\end{corollary}
\begin{proof}
	We have $Gf= \lim\limits_{n\to +\infty} G_{B_n}f$. Hence by the preceding proposition $Gf$ is
	radial on $NA$ as well.
\end{proof}

We finish this section by proving that
if $\varphi $ is radial then  
$U^{\varphi}_B$ is radial as well. The statement holds both for the Euclidean space and $NA$ and it follows from the fact that the Green operator preserves radiality.
More precisely, we have the following lemma. 
\begin{lemma}\label{radiality-in-bounded}
	We assume that $\varphi$ satisfies $(H_1)$-$(H_3)$ and radial with respect to the first variable.
	Then for every $c>0$ there exists a unique $u\in
	\mathcal{C}(\overline{B})$ radial on $B$ such that \begin{equation}\label{radial-solution}
	\left\{
	\begin{array}{ll}
	\L u=\varphi(\cdot,u), & \hbox{in  $B$;} \\
	u=c, & \hbox{on  $\partial B$,}
	\end{array}
	\right.
	\end{equation}
	Furthermore, $$u+G_B(\varphi(\cdot,u))=c, \hbox{ in }B.$$ 
\end{lemma}

\begin{proof}
	The proof is the same as in \cite{Ghardallou2}.\\ 
	Let $\alpha=\min\{c-G_B(\varphi(\cdot,c))(s), s\in NA\}$ and
	$$\mathcal{K}=\{u\in\mathcal{C}(\overline{B}),~~~~ \alpha \leq u \leq c,~~~~u \hbox{ is radial in }B \}.$$
	$\mathcal{K}$ is nonempty, closed for the uniform norm and convex. We
	consider \begin{equation*}
	\begin{aligned}
	T&: & \mathcal{C}(\overline{B})\to&  \mathcal{C}(\overline{B}) \\
	&  &  u\mapsto &c-G_B(\varphi(\cdot,u)).
	\end{aligned}
	\end{equation*}
	and we prove that $T(\mathcal{K})\subset \mathcal{K}$. Then Schauder theorem implies that $T$ has a fixed point in $\mathcal{K}$ i.e. a radial positive solution of \eqref{radial-solution}.
\end{proof}
\section{{ Harnack-type inequality}}\label{Harnack-type-ine}
In this section we will prove Theorem \ref{Mainthe-Harnack-ineq} - the Harnack-type inequality \footnote{See introduction.}. 
The proof is based on the following three lemmas. First we need to compare the Green function for
$L$ with the one for the Laplace operator $\Delta $ on balls $B_r(a)=\{ x\in \R ^d: \| x-a\| <r\} $ of small radii contained in
$\Omega $. More precisely, we intend to compare $G_{B_r}^L$ with $G_{B_r}^{\Delta}$ uniformly,
provided the centers of the balls do not approach $\partial \Omega $.

\begin{lemma}\label{comparaison-Green-functions}
	Let $\overline D$ be a compact subset of $ \Omega $ such that $\mbox{dist}(\overline D,\partial
	\Omega )=2r_0$ for some $r_0>0$. There is a constant $M>0$
	such that for every $r\leq r_0$ and
	for every ball $B_r=B(a,r)$, $a\in \overline D$
	\begin{equation}\label{comparison}
	M^{-1} G_{B_r}^{\Delta}(x,y)\leq G_{B_r}^L(x,y)\leq M \, G_{B_r}^{\Delta}(x,y), \quad x,y\in B_r,\end{equation}
	where $G_{B_r}^L$ is the Green function for $L$ in $B_r$.
\end{lemma}

The next lemma is the Harnack inequality for positive $L$-harmonic functions in $B(a,r)$ with a
constant independent of $a$ and $r$.

\begin{lemma}\label{constant-harnack-indepandence}
	Under the same hypotheses as in Lemma \ref{comparaison-Green-functions},  there exists a constant
	$M_1>0$ such that for every $r\in[0,r_0[$ and for every positive $L$-harmonic function $h$ in
	$B(a,r)$
	\begin{equation}\label{harnack-indep}
	\sup_{x\in B(a,\frac{r}{4})} h(x) \leq M_1
	\inf_{x\in B(a,\frac{r}{4})} h(x). \end{equation}
	
\end{lemma}

Finally, we need the following control of potentials:

\begin{lemma}\label{comparaison-function-and-her-potential}
	
	Let $p \in \mathcal{K}_d^{loc}(\Omega)$, {$p \geq 0$} and let $s$ be a positive super\-harmonic function 
	Then, under the same hypotheses as in Lemma
	\ref{comparaison-Green-functions}, for every $r\in[0,r_0]$ and $a\in \overline{D}$
	$$ G_{B_r}^L(p s)(x)\leq 2M^3c_d s(x) \sup_{y\in {B_r}} \int_{B_r} G_{\mathbb{R}^d}^{\Delta}(y,z)p(z)dz,  \hbox{  $x\in {B_r}$},$$
	where ${B_r}=B(a,r)$ and $c_d$ is a constant depending only on the dimension as in Lemma \ref{3-G
		theorem for elliptic operator} below.
\end{lemma}

\subsection{Proof of Theorem \ref{Mainthe-Harnack-ineq}}
Now we are ready to prove Theorem \ref{Mainthe-Harnack-ineq}. In fact we prove a stronger
statement - a uniform Harnack-type inequality for a family of small balls. Then a standard
compactness argument will do.
We follow the ideas of \cite{Boukricha}.

\begin{theorem}\label{generalized-Harnack-ineq}
	Under the same hypotheses as in Theorem \ref{Mainthe-Harnack-ineq} and Lem\-ma
	\ref{comparaison-Green-functions}, there exist a constant $C>0$ such that for every $r\leq r_0$,
	every ball $B=B(a,r)$,  $a\in \overline D$ and every positive solution of \eqref{Main-problem} in
	$B$ we have
	\begin{equation}\label{uniform-Har}
	\sup_{x\in B(a, \frac{r}{4})}u(x)\leq C (1+\inf_{x\in B(a,\frac{r}{4})}u(x)).
	\end{equation}
\end{theorem}

\begin{proof}
	
	
	Let $r_0$ be such that the assumptions of Lemma \ref{comparaison-Green-functions} are satisfied. We
	can impose in addition that $r_0\leq 1$. Let $M$ the constant in \eqref{comparison}. We have
	$$\lim_{r\to 0}\sup_{x\in B(a,r)} \int_{B(a,r)} G_{\mathbb{R}^d}^{\Delta}(x,z)p(z)\,dz=0,$$
	where $p$ is as in $(H'_1)$. Hence making $r_0$ possibly smaller we may conclude that for every $r\in ]0,r_0]$
	\begin{equation}\label{Kato-c-1}
	\max (2M^3c_d, M\slash 2) \sup_{x\in B} \int_B G_{\mathbb{R}^d}^{\Delta}(x,z)p(z)\,dz\leq 1/2,
	\end{equation}
	Let $x,y \in B(a,\frac{r}{4})\subset B=B(a,r)$ and $u$ be a positive solution of
	\eqref{Main-problem} in $B(a,r)$. Then
	$$\begin{aligned}[t]
	u(x)=H_Bu(x)-G_B^L(\varphi(\cdot,u))(x)
	&\leq H_Bu(x).
	\end{aligned}$$
	Let now $M_1$ be the constant in \eqref{harnack-indep}. Since $H_Bu$ is an $L$-harmonic function in
	$B(a,r)$ then for every $x,y\in B(a,\frac{r}{4})$
	\begin{equation}\label{Har1}
	H_Bu(x) \leq  M_1 H_Bu(y).
	\end{equation}
	In addition, by $(H'_1)$
	$$\begin{aligned}[t]
	H_Bu(y)=u(y)+G_B^L(\varphi(\cdot,u))(y)
	&\leq u(y)+G_B^L(\varphi(\cdot,H_Bu))(y)\\
	&\leq u(y)+G_B^L(p(H_Bu+1)(y)\\
	&\leq u(y)+G_B^L(pH_Bu)(y)+G_B^L(p)(y)
	\end{aligned}$$
	Moreover, by Lemma \ref{comparaison-function-and-her-potential} $$ G_B^L(pH_Bu)(y)\leq 2M^3c_d
	H_Bu(y) \sup_{x\in B} \int_B G_{\mathbb{R}^d}^{\Delta}(x,z)p(z)\,dz,$$ So by \eqref{Kato-c-1} $$
	H_Bu(y)\leq u(y)+1/2 H_Bu(y)+G_B^L(p)(y).$$ Multiplying both sides by 2, we get:
	$$ H_Bu(y)\leq 2u(y)+2G_B^L(p)(y).$$
	Furthermore, applying Lemma \ref{comparaison-Green-functions} and \eqref{Kato-c-1} again, we have $$ G_B^L(p)(y) \leq M G_B^{\Delta}(p)(y) \leq
	1.$$ Consequently
	$$ H_Bu(y) \leq  2u(y) + 2.$$ By \eqref{Har1} we conclude
	
	$$ u(x) \leq 2M_1(u(y)+1), \hbox{   $x,y\in B(a,\frac{r}{4})$}.$$
\end{proof}

\subsection{Proofs of lemmas }
To obtain a uniform comparison \eqref{comparison} for small balls we will use Theorem \ref{Theorem
	Hueber Sieveking} below. Suppose we have a family {$\mathcal{F}_{\lambda, \alpha }$ of second order
	strictly elliptic operators on $\R ^d$ having $\alpha$-H\"older coefficients and $\lambda$ the
	ellipticity constant.} By the $\alpha $ - H\"older norm of a bounded function $f$ we mean here
$$
\| f\| _{\alpha }=\sup _{x\in R^d}|f(x)| + \sup _{x\neq y\in R^d}\frac{|f(x)-f(y)|}{|x-y|^{\alpha}}.
$$
We assume that {$0<\lambda <1$}. $\mathcal{L}\in \mathcal{F}_{\lambda, \alpha }$ is of the form
$$\mathcal{L}=\sum _{i,j}a_{i,j}(x) \partial_i \partial _j+\sum_i  b_i(x) \partial_i+c(x),$$ where $a_{ij}=a_{ji}$,
for every $x, \xi\in\mathbb{R}^d$ $$ a_{ij}(x)\xi_i\xi_j\geq \lambda |\xi|^2$$
and
$$
\| a_{ij}\| _{\alpha }, \| b_{i}\| _{\alpha }, \| c\| _{\alpha }\leq \lambda ^{-1}.$$

The main ingredient in the proof of Lemma \ref{comparaison-Green-functions}  is the following
theorem

\begin{theorem}\label{Theorem Hueber Sieveking}(\cite{Hueber/Sieveking})
	Suppose $\lambda $ is fixed.  For every bounded $\mathcal{C}^{1,1}$ domain $D$ there exists a
	constant $C>0$ such that
	$$ C^{-1}G_D^{\Delta}\leq G_D^{\mathcal{L}}\leq C G_D^{\Delta}, \hbox{  on }D\times D. $$
	for every $\mathcal{L}\in \mathcal{F}_{\lambda,\alpha }$.
\end{theorem}
\begin{proof}[Proof of Lemma \ref{comparaison-Green-functions}]
	Without loss of generality we may assume that $L\in \mathcal{F}_{\lambda,\alpha }$. Indeed, $L$,
	defined originally on $\Omega $, may be extended to an operator $\tilde L \in
	\mathcal{F}_{\lambda,\alpha
	}$ in the way that $\tilde L = L$ in $\displaystyle\bigcup _{a\in
		\overline D}B(a,r_0)$.
	
	Fix $a_0\in \overline D$. For $B_1=B(a_0,r_0)$, \eqref{comparison} follows directly from Theorem
	\ref{Theorem Hueber Sieveking}. 
	To get the same statement for $B(a,rr_0)$ one needs to consider
	the family of operators 
	\begin{equation}\label{operator}
	L_r=\sum _{i,j}a_{i,j}(\psi_r(x))\partial_i \partial_j+ r\sum _{i}b_i(\psi_r(x))\partial_i+r^2c(\psi_r(x)), \quad r\leq 1 \end{equation}
	on $B_1=B(a_0,r_0)$ where \begin{equation*} \psi_r: B_1\to B_r=B(a,rr_0)\quad \mbox{is given
		by}\quad \psi _r(x)= a+r(x-a_0).
	\end{equation*} Since $L_r\in \mathcal{F}_{\lambda,\alpha }$, the
	estimate in \eqref{comparison} on $B_1$ for the whole class of operators $L_r$  is uniform, i.e.
	\begin{equation}\label{L-r-comparison}
	C^{-1} G_{B_1}^{\Delta}(x,y)\leq G_{B_1}^{L_r}(x,y)\leq C \, G_{B_1}^{\Delta}(x,y), \quad x,y\in B_1.\end{equation}
	In addition, if $u$ is a $L$-harmonic function on $B_r$ then $u\circ\psi_r$ is $L_r$-harmonic on
	$B_1$ i.e. $L_r(u\circ\psi_r)(x)=r^2(Lu)(\psi_r(x)),$ $x\in B_1$. It follows that
	$$G^L_{B_r}(x,y)=r^{-d+2}G^{L_r}_{B_1}(\psi^{-1}_r(x),\psi^{-1}_r(y)), \hbox{ }x,y\in B_r,$$
	where \begin{equation*}
	\psi_r^{-1}: B_r\to B_1\\
	\end{equation*} is the inverse function to $\psi _r$. Moreover, $\Delta_r=\Delta$, so we have the same for $G_{B_r}^{\Delta}$. Finally, replacing $G^{L_r}_{B_1}$, $G^{\Delta }_{B_1}$ in \eqref{L-r-comparison} by $G^{L}_{B_r}$, $G^{\Delta }_{B_r}$ we get
	\eqref{comparison}.
\end{proof}

Lemma \ref{constant-harnack-indepandence} follows from the classical Harnack inequality for
elliptic operators, see \cite{Gilbarg/Trudinger}, p.189. To get the constant independent of $r$ we
use the same family of operators $L_r$.

Due to Lemma \ref{comparaison-Green-functions} we are able to conclude a version of 3-G theorem for
$L$ adopted to our setting:

\begin{lemma}\label{3-G theorem for elliptic operator}
	Suppose that the assumptions of Lemma \ref{comparaison-Green-functions} are satisfied and
	$B=B(a,r)$. Then for every open ball $B$ in $\mathbb{R}^d$ and a Radon positive measure $\nu$, we
	have
	\begin{equation}\label{3G}
	\int_B G_B^L(x,z)G_B^L(z,y)\nu(dz)\leq 2M^3 c_d G_B^{L}(x,y)\gamma _{\nu }(B),
	\end{equation}
	where
	$$\gamma _{\nu }(B)=\sup_{x\in B}\int_B G_{\mathbb{R}^d}^{\Delta}(x,z)\nu(dz)$$
	and $c_d$ is a constant depending only on the dimension.
\end{lemma}
\noindent \eqref{3G} follows directly from an analogous statement for $G_B^{\Delta}$, see \cite{B.H.H},
p.132:
\begin{equation*}
\int_B G_B^{\Delta}(x,z)G_B^{\Delta}(z,y)\nu(dz)
\leq 2c_d G_B^{\Delta}(x,y) \gamma _{\nu}(B).
\end{equation*}
Finally, we are able to prove Lemma \ref{comparaison-function-and-her-potential}
\begin{proof}
	Let $\lambda$ a positive Radon measure and $q(z)=\int _BG_B^L(z,y) \lambda(dy)$ By Lemma \ref{3-G
		theorem for elliptic operator} applied to $ p(z)\,dz$ instead of $\nu(dz)$, we obtain
	$$\begin{aligned}[t]
	&G_B^L(pq)(x)=\int_B G_B^L(x,z)p (z)\Big (\int_B G_B^L(z,y)  \lambda(dy)\Big ) dz\\
	&=\int_B \int_B G_B^L(x,z)G_B^L(z,y) p(z)dz \lambda(dy)\\
	&\leq 2M^3 c_d \int _BG_B^L(x,y)\sup_{x\in B}( \int_B G_{\mathbb{R}^d}^{\Delta}(x,z) p(z)dz)\lambda(dy)\\
	&=2M^3 c_d q(x) \Big (\sup_{x\in B} \int_B G_{\mathbb{R}^d}^{\Delta}(x,z) p(z)dz\Big )
	\end{aligned}$$
	Or $s$ is the limit of locally integrable increasing potentials $(s_n)$ in $B$ which can be written
	$$s_n(x)= \int_{B} G_{B}^L(x,y)\lambda_n(dy), \hbox{ where $\lambda_n $ is a positive Radon
		measure}. $$ Then by integration with respect to $dy$, we can conclude for every $n\in\mathbb{N}$,
	$$ G_B^L(p s_n)(x)\leq 2M^3c_d s_n(x) \sup_{x\in B} \int_B G_{\mathbb{R}^d}^{\Delta}(x,z) p(z)(dz), $$
	and by the monotone convergence theorem we deduce the result.
	
\end{proof}


\end{document}